\documentclass{amsart}
\usepackage{amsmath,amssymb,amsthm,amscd,amsfonts}

\theoremstyle{plain}
\newtheorem{thm}{Theorem}[section]
\newtheorem{prop}[thm]{Proposition}

\newtheorem{lem}[thm]{Lemma}

\newtheorem{main}{Main Theorem}

\newtheorem{prm}{Problem}

\theoremstyle{remark}

\theoremstyle{definition}

\newcommand{\N}{\mathbb{N}}                       % natural numbers
\newcommand{\I}{\mathbf{I}}                       % unit interval
\newcommand{\cl}{\operatorname{cl}}               % closure
\newcommand{\card}{\operatorname{card}}	          % cardinality
\newcommand{\diam}{\operatorname{diam}}           % diameter
\newcommand{\fin}{\operatorname{Fin}}             % finite sets
\newcommand{\cpt}{\operatorname{Comp}}             % compact sets

\title[The compact strong $Z$-set property in a hyperspace of finite subsets]{The compact strong $Z$-set property in a hyperspace of finite subsets}

\author[K.~Koshino]{Katsuhisa Koshino}

\address{Division of Mathematics, Pure and Applied Sciences, University of Tsukuba, Tsukuba, 305-8571, Japan}

\email{kakoshino@math.tsukuba.ac.jp}



\begin{document}

\begin{abstract}
Let $X$ be a non-degenerate, connected, locally path-connected metrizable space and $\fin(X)$ be the hyperspace consisting of non-empty finite subsets in $X$ endowed with the Vietoris topology.
In this paper, we show that every compact set in $\fin(X)$ is a strong $Z$-set.
\end{abstract}

\maketitle

\section{Introduction}

Throughout this paper, spaces are metrizable and maps are continuous.
A closed subset $A$ of a space $X$ is said to be a \textit{(strong) $Z$-set} in $X$ if for each open cover $\mathcal{U}$ of $X$,
 there exists a map $f : X \to X$ such that $f$ is $\mathcal{U}$-close to the identity map on $X$ and the (closure of) image misses $A$.
We recall that for maps $f : X \to Y$ and $g : X \to Y$, and for an open cover $\mathcal{U}$ of $Y$, $f$ is $\mathcal{U}$-close to $g$ if for each $x \in X$, there exists a member $U \in \mathcal{U}$ such that the both $f(x)$ and $g(x)$ are contained in $U$.
These notions play central roles in the theory of infinite-dimensional topology.
It is said that a space $X$ has \textit{the compact (strong) $Z$-set property} if every compact set in $X$ is a (strong) $Z$-set.
This property is very important because typical infinite-dimensional manifolds have the compact strong $Z$-set property.
Given a space $X$, let $\fin(X)$ be the hyperspace of non-empty finite subsets of $X$ endowed with the Vietoris topology.
D.~Curtis \cite[Proposition~7.3]{Cu4} proved that if $X$ is non-degenerate, connected, locally path-connected and {\bf $\sigma$-compact},\footnote{A space is $\sigma$-(locally) compact provided that it is a countable union of (locally) compact subsets.}
 then $\fin(X)$ has the compact strong $Z$-set property.
In the case that $X$ is not separable, M.~Yaguchi \cite[Proposition~6.1]{Yag} showed that if $X$ is a normed linear space of dimension $\geq 1$,
 then $\fin(X)$ has the compact strong $Z$-set property.
In this paper, we generalize these results as follows:

\begin{main}
Let $X$ be non-degenerate, connected and locally path-connected.
Then $\fin(X)$ has the compact strong $Z$-set property.
\end{main}

\section{Preliminaries}

In this section, we fix some notation and introduce some lemmas concerning nice subdivisions of simplicial complexes used in the next section.
We denote the set of natural numbers by $\N$ and the closed unit interval by $\I$.
Let $Y = (Y,\rho)$ be a metric space.
For a point $y \in Y$ and a subset $A \subset Y$, we define the distance $\rho(y,A)$ between $y$ and $A$ by $\rho(y,A) = \inf\{\rho(y,a) \mid a \in A\}$.
For $\epsilon > 0$, let $B_\rho(y,\epsilon) = \{y' \in Y \mid \rho(y,y') < \epsilon\}$, $\overline{B_\rho(y,\epsilon)} = \{y' \in Y \mid \rho(y,y') \leq \epsilon\}$, $N_\rho(A,\epsilon) = \{y' \in Y \mid \rho(y',A) < \epsilon\}$ and $\overline{N_\rho(A,\epsilon)} = \{y' \in Y \mid \rho(y',A) \leq \epsilon\}$.
The diameter of $Y$ is denoted by $\diam_\rho{Y}$.
Let $\cpt(Y) = (\cpt(Y),\rho_H)$ be the hyperspace consisting of compact subsets of $Y$ with the Hausdorff metric $\rho_H$ induced by $\rho$ defined as follows:
 $$\rho_H(A,B) = \inf\{r > 0 \mid A \subset N_\rho(B,r), B \subset N_\rho(A,r)\}.$$
Note that $\fin(Y)$ is regarded as a subspace of $\cpt(Y)$.

Given a simplicial complex $K$, we denote the polyhedron\footnote{In this paper, we do not need polyhedra to be metrizable.} of $K$ by $|K|$ and the $n$-skeleton of $K$ by $K^{(n)}$ for each $n \in \N \cup \{0\}$.
Regarding $\sigma \in K$ as a simplicial complex consisting of its faces, we write $\sigma^{(n)}$ as the set of $i$-faces of $\sigma$, $i \leq n$.
The boundary of a simplex $\sigma$ is denoted by $\partial{\sigma}$.
The next two lemmas are used in the proof of Theorem~E in \cite{Cu1}.

\begin{lem}\label{subd.1}
Let $Y = (Y,\rho)$ be a metric space, $K$ a simplicial complex and $f : |K| \to Y$ a map.
For each map $\alpha : Y \to (0,\infty)$, there exists a subdivision $K'$ of $K$ such that $\diam_\rho{f(\sigma)} < \inf_{x \in \sigma} \alpha f(x)$ for all $\sigma \in K'$.
\end{lem}

\begin{proof}
By induction, we shall construct subdivisions $K_n$ of the $n$-skeleton $K^{(n)}$ for all $n \in \N \cup \{0\}$ so that $K_n \subset K_{n+1}$ and $\diam_\rho{f(\sigma)} < \inf_{x \in \sigma} \alpha f(x)$ for every $\sigma \in K_n$.
Then $K' = \bigcup_{n \in \N \cup \{0\}} K_n$ will be the desired subdivision of $K$.
Let $K_0 = K^{(0)}$.
Assume that $K_n$ has been constructed.
Take any $\sigma \in K^{(n+1)} \setminus K^{(n)}$ and let $\delta = \inf_{x \in \sigma} \alpha f(x)/2 > 0$.
By the inductive assumption, we have $\diam_\rho{f(\tau)} < \inf_{x \in \tau} \alpha f(x)$ for every $\tau \in \{\tau' \in K_n \setminus K_{n-1} \mid \tau' \subset \sigma\}$.
It follows from the continuity of $\alpha$ and the compactness of $\tau$ that there is $0 < \delta_\tau < (\inf_{x \in \tau} \alpha f(x) - \diam_\rho{f(\tau)})/3$ such that for each $y \in N_\rho(f(\tau),\delta_\tau)$, $\alpha(y) > (2\inf_{x \in \tau} \alpha f(x) + \diam_\rho{f(\tau)})/3$.
Then we have the open cover
 $$\{(f|_\sigma)^{-1}(B_\rho(y,\delta)) \setminus \partial{\sigma} \mid y \in Y\} \cup \{(f|_\sigma)^{-1}(N_\rho(f(\tau),\delta_\tau)) \mid \tau \in K_n \setminus K_{n-1} \text{ and } \tau \subset \sigma\}$$
 of $\sigma$.
By the same argument as \cite[Proposition~4.7.10]{Sa6}, we can find a triangulation $K_\sigma$ of $\sigma$ such that $\{\tau \in K_n \mid \tau \subset \sigma\} \subset K_\sigma$ and the following conditions are satisfied:
\begin{itemize}
 \item $\sigma' \subset (f|_\sigma)^{-1}(B_\rho(y,\delta))$ for some $y \in Y$ if $\sigma' \in K_\sigma$ and $\sigma' \cap \partial{\sigma} = \emptyset$;
 \item $\sigma' \subset (f|_\sigma)^{-1}(N_\rho(f(\tau),\delta_\tau))$ if $\sigma' \in K_\sigma$ and $\tau \subset \sigma'$ for some $\tau \in K_n \setminus K_{n-1}$.
\end{itemize}
When $\sigma' \in K_\sigma$ and $\sigma' \cap \partial{\sigma} = \emptyset$,
 $f(\sigma') \subset B_\rho(y,\delta)$,
 and hence
 $$\diam_\rho{f(\sigma')} < 2\delta = \inf_{x \in \sigma} \alpha f(x) \leq \inf_{x \in \sigma'} \alpha f(x).$$
When $\sigma' \in K_\sigma$ and $\tau \subset \sigma'$ for some $\tau \in K_n \subset K_{n-1}$,
 $f(\sigma') \subset N_\rho(f(\tau),\delta_\tau))$,
 so
 $$\diam_\rho{f(\sigma')} < \diam_\rho{f(\tau)} + 2\delta_\tau < (2\inf_{x \in \tau} \alpha f(x) + \diam_\rho{f(\tau)})/3 \leq \inf_{x \in \sigma'} \alpha f(x).$$
Let $K_{n+1} = \{K_\sigma \mid \sigma \in K^{(n+1)} \setminus K^{(n)}\}$.
Thus the proof is complete.
\end{proof}

\begin{lem}\label{subd.2}
For each map $\alpha : |K| \to (0,\infty)$ of the polyhedron of a simplicial complex $K$ and $\beta > 1$, there is a subdivision $K'$ of $K$ such that $\sup_{x \in \sigma} \alpha(x) < \beta\inf_{x \in \sigma} \alpha(x)$ for any $\sigma \in K'$.
\end{lem}

\begin{proof}
For each $x \in |K|$, we can choose an open neighborhood $U(x)$ of $x$ in $|K|$ so that if $y \in U(x)$,
 then $|\alpha(x) - \alpha(y)| < (\beta - 1)\alpha(x)/(\beta + 1)$.
Then $\mathcal{U} = \{U(x) \mid x \in |K|\}$ is an open cover of $|K|$.
According to Theorem~4.7.11 of \cite{Sa6}, there is a subdivision $K'$ of $K$ that refines $\mathcal{U}$.
Take any simplex $\sigma \in K'$ and any point $y \in \sigma$.
By the compactness of $\sigma$, we can find $z \in \sigma$ such that $\alpha(z) = \inf_{z' \in \sigma} \alpha(z')$.
Since $K'$ refines $\mathcal{U}$,
 there exists a point $x \in |K|$ such that $\sigma \subset U(x)$.
Then $|\alpha(x) - \alpha(y)| < (\beta - 1)\alpha(x)/(\beta + 1)$ and $|\alpha(x) - \alpha(z)| < (\beta - 1)\alpha(x)/(\beta + 1)$.
Observe that
 $$\alpha(y) < 2\beta\alpha(x)/(\beta + 1) < \beta\alpha(z),$$
 which implies that $\sup_{z' \in \sigma} \alpha(z') < \beta\inf_{z' \in \sigma} \alpha(z')$.
Hence $K'$ is the desired subdivision.
\end{proof}

\section{The compact strong $Z$-set property of $\fin(X)$}

This section is devoted to proving the main theorem.
From now on, we use an admissible metric $d$ on a space $X$ and the Hausdorff metric $d_H$ induced by $d$ on the hyperspace $\fin(X)$.
Combining Lemmas~2.3, 3.6, and the proof of Theorem~2.4 of \cite{CN} (cf.~\cite[Proposition~3.1]{Yag}), we have the following proposition:

\begin{prop}\label{ar}
The hyperspace $\fin(X)$ is an AR if and only if $X$ is connected and locally path-connected.
\end{prop}

\begin{lem}\label{sph.}
If $X$ is non-degenerate and connected,
 then for each $x \in X$ and $0 < \epsilon < \diam_d{X}/4$, there exists a point $y \in X$ such that $d(x,y) = \epsilon$.
\end{lem}

\begin{proof}
Suppose the contrary.
Then $X$ can be separated by disjoint non-empty open subsets $B_d(x,\epsilon)$ and $X \setminus \overline{B_d}(x,\epsilon)$,
 which contradicts to the connectedness of $X$.
Thus the proof is complete.
\end{proof} 

\begin{lem}\label{subseq.}
Suppose that $\{A_n\}_{n \in \N}$ is a sequence in $\fin(X)$ converging to $A \in \fin(X)$.
Then for each $B_n \subset A_n$, $\{B_n\}_{n \in \N}$ has a subsequence converging to some $B \subset A$.
\end{lem}

\begin{proof}
According to Lemma~1.11.2.~(3)\footnote{This holds without the assumption that $X$ is separable.} of \cite{Mil3}, $\tilde{A} = A \cup \bigcup_{n \in \N} A_n$ is compact.
Hence the hyperspace $\cpt(\tilde{A}) = (\cpt(\tilde{A}),(d|_{\tilde{A} \times \tilde{A}})_H)$ is compact, see \cite[Theorem~5.12.5.~(3)]{Sa6},
 which implies that $\{B_n\}_{n \in \N}$ has a subsequence $\{B_{n_i}\}_{i \in \N}$ converging to some $B \in \cpt(\tilde{A})$.
By Lemma~1.11.2.~(2)\footnotemark[3] of \cite{Mil3}, we have
\begin{multline*}
 B = \{x \in X \mid \text{for each } i \in \N, \text{ there is } b_{n_i} \in B_{n_i} \text{ such that } \lim_{i \to \infty} b_{n_i} = x\}\\
 \subset \{x \in X \mid \text{for each } i \in \N, \text{ there is } a_{n_i} \in A_{n_i} \text{ such that } \lim_{i \to \infty} a_{n_i} = x\} = A.
\end{multline*}
Thus the proof is complete.
\end{proof}

\begin{lem}\label{arc}
Let $\alpha : \fin(X) \to (0,\infty)$ be a map.
If $X$ is locally path-connected,
 then there exsits a map $\beta : \fin(X) \to (0,\infty)$ such that for any $A \in \fin(X)$, each point $x \in \overline{N_d}(A,\beta(A))$ has an arc $\gamma : \I \to X$ from some point of $A$ to $x$ of $\diam_d{\gamma(\I)} < \alpha(A)$.
\end{lem}

\begin{proof}
For each $A \in \fin(X)$, let
 $$\Xi(A) = \left\{
  \begin{array}{l|l}
  \eta > 0 &\left.
  \begin{array}{ll}
  \text{there exists } 0 < \epsilon < \alpha(A) \text{ such that for any } a \in A \text{ and}\\
  x \in \overline{B_d}(a,\eta), \text{ there is an arc from } a \text{ to } x \text{ of diameter } < \epsilon
  \end{array}
  \right.
  \end{array}\right\}$$
 and $\xi(A) = \sup\Xi(A)$.
Note that $\Xi(A) \neq \emptyset$ for all $A \in \fin(X)$.
Indeed, let $0 < \epsilon < \alpha(A)$.
Since $X$ is locally path-connected,
 and hence locally arcwise-connected \cite[Corollary~5.14.7]{Sa6},
 for each $a \in A$, there exists $\eta(a) > 0$ such that for any $x \in \overline{B_d}(a,\eta(a))$, $a$ and $x$ are connected by an arc of diameter $< \epsilon$.
Then $\eta = \min_{a \in A}\eta(a) \in \Xi(A)$.
By the definition, $\xi(A) \leq \alpha(A)$. 

We shall show that $\xi$ is lower semi-continuous.
Take any $t \in (0,\infty)$ and any $A \in \xi^{-1}((t,\infty))$.
Then we can choose $t < \eta \leq \xi(A)$ so that there is $0 < \epsilon < \alpha(A)$ such that for any $a \in A$ and any $x \in \overline{B_d}(a,\eta)$, $a$ and $x$ are connected by an arc of diameter $< \epsilon$.
Since $X$ is locally arcwise-connected,
 there exists $\delta_1 > 0$ such that any $a \in A$ and any $x \in \overline{B_d}(a,\delta_1)$ are connected by an arc of diameter $< (\alpha(A) - \epsilon)/2$.
By the continuity of $\alpha$, we can find $\delta_2 > 0$ such that for each $B \in B_{d_H}(A,\delta_2)$, $|\alpha(A) - \alpha(B)| < (\alpha(A) - \epsilon)/2$.
Let $\delta = \min\{\delta_1, \delta_2, (\eta - t)/2\}$ and $B \in B_{d_H}(A,\delta)$.
Observe that $(\alpha(A) + \epsilon)/2 < \alpha(B)$.
Fix any $b \in B$ and any $x \in \overline{B_d}(b,(\eta + t)/2)$.
Since $d_H(A,B) < \delta$,
 we can take $a \in A$ such that $d(a,b) < \delta \leq \delta_1$,
 and hence there exists an arc $\gamma_1$ from $b$ to $a$ of diameter $< (\alpha(A) - \epsilon)/2$.
On the other hand,
 $$d(a,x) \leq d(a,b) + d(b,x) < \delta + (\eta + t)/2 \leq (\eta - t)/2 + (\eta + t)/2 = \eta,$$
 which implies that there is an arc $\gamma_2$ from $a$ to $x$ of diameter $< \epsilon$.
Joining these arcs $\gamma_1$ and $\gamma_2$, we can obtain an arc from $b$ to $x$ of diameter $< (\alpha(A) - \epsilon)/2 + \epsilon = (\alpha(A) + \epsilon)/2 < \alpha(B)$.
Hence $t < (\eta + t)/2 \leq \xi(B)$,
 which means that $\xi$ is lower semi-continuous.

According to Theorem~2.7.6 of \cite{Sa6}, we can find a map $\beta : \fin(X) \to (0,\infty)$ such that $0 < \beta(A) < \xi(A)$ for all $A \in \fin(X)$,
 that is the desired map.
\end{proof}

The next lemma is useful to detect a strong $Z$-set in an ANR.

\begin{lem}[Lemma~7.2 of \cite{Cu4}]\label{str.Z}
Let $A$ be a topologically complete, closed subset of an ANR $Y$.
If $A$ is a countable union of strong $Z$-sets in $Y$,
 then it is a strong $Z$-set.
\end{lem}

We denote the cardinality of a set $A$ by $\card{A}$.
For each $k \in \N$, let $\fin^k(X) = \{A \in \fin(X) \mid \card{A} \leq k\}$.
As is easily observed, $\fin^k(X)$ is closed in $\fin(X)$.
Applying the above lemma~\ref{str.Z}, we only need to show the following proposition for proving the main theorem.

\begin{prop}
Suppose that $X$ is non-degenerate, connected and locally path-connected.
Then for each $k \in \N$, $\fin^k(X)$ is a strong $Z$-set in $\fin(X)$.
\end{prop}

\begin{proof}
Let $\mathcal{U}$ be an open cover of $\fin(X)$ and $k \in \N$.
We shall construct a map $\phi : \fin(X) \to \fin(X)$ so that $\phi$ is $\mathcal{U}$-close to the identity map on $\fin(X)$ and $\cl{\phi(\fin(X))} \cap \fin^k(X) = \emptyset$,
 where for a subset $\mathcal{A} \subset \fin(X)$, $\cl{\mathcal{A}}$ means the closure of $\mathcal{A}$ in $\fin(X)$.
Take an open cover $\mathcal{V}$ of $\fin(X)$ that is a star-refinement of $\mathcal{U}$.
Since $\fin(X)$ is an AR by Proposition~\ref{ar},
 there are a simplicial complex $K$ and maps $f : \fin(X) \to |K|$, $g : |K| \to \fin(X)$ such that $gf$ is $\mathcal{V}$-close to the identity map on $\fin(X)$, refer to \cite[Theorem~6.6.2]{Sa6}.
It remains to show that there exists a map $h : |K| \to \fin(X)$ $\mathcal{V}$-close to $g$ such that $\cl{h(|K|)} \cap \fin^k(X) = \emptyset$ because $\phi = hf$ will be the desired map.

Take a map $\alpha : \fin(X) \to (0,\min\{1,\diam_d{X}\})$ so that the family $\{B_{d_H}(A,2\alpha(A)) \mid A \in \fin(X)\}$ refines $\mathcal{V}$.
Since $X$ is locally path-connected,
 according to Lemma~\ref{arc}, there is a map $\beta : \fin(X) \to (0,\infty)$ such that for any $A \in \fin(X)$, each point $x \in \overline{N_d}(A,\beta(A))$ has an arc $\gamma : \I \to X$ from some point of $A$ to $x$ of $\diam_d{\gamma(\I)} < \alpha(A)/2$.
We may assume that $\beta(A) \leq \alpha(A)/2$ for every $A \in \fin(X)$.
Combining Lemmas~\ref{subd.1} with \ref{subd.2}, we can replace $K$ with a subdivision so that for each $\sigma \in K$,
\begin{enumerate}
 \item $\diam_{d_H}{g(\sigma)} < \inf_{y \in \sigma} \beta g(y)/2$,
 \item $\sup_{y \in \sigma} \beta g(y) < 2\inf_{y \in \sigma} \beta g(y)$,
 \item $\sup_{y \in \sigma} \alpha g(y) < 4\inf_{y \in \sigma} \alpha g(y)/3$.
\end{enumerate}

For every $v \in K^{(0)}$, fix a point $x(v) \in g(v)$.
According to Lemma~\ref{sph.}, we can find a point $z(v,j) \in X$ with $d(x(v),z(v,j)) = j\beta g(v)/(4(k+1))$ for each $j = 0, \cdots, k$.
Let $h(v) = g(v) \cup \{z(v,j) \mid j = 0, \cdots, k\}$.
Clearly, $\card{h(v)} \geq k + 1$ and $d_H(g(v),h(v)) \leq \beta g(v) \leq \alpha g(v)/2$.
Observe that for any $0 \leq i < j \leq k$,
\begin{align*}
 d(z(v,i),z(v,j)) &\geq |d(x(v),z(v,i)) - d(x(v),z(v,j))| = (j-i)\beta g(v)/(4(k+1))\\
 &\geq \beta g(v)/(4(k+1)).
\end{align*}

Next, we will extend $h$ over $|K^{(1)}|$.
Let $\sigma \in K^{(1)} \setminus K^{(0)}$, $\sigma^{(0)} = \{v_1, v_2\}$ and $\hat\sigma$ be the barycenter of $\sigma$.
Due to conditions (1) and (2), we have for any $m = 1, 2$ and $j = 0, \cdots, k$,
\begin{align*}
 d(z(v_m,j),g(\hat\sigma)) &\leq d(z(v_m,j),g(v_m)) + d_H(g(v_m),g(\hat\sigma))\\
 &\leq d(z(v_m,j),g(v_m)) + \diam_{d_H}{g(\sigma)}\\
 &< d(z(v_m,j),x(v_m)) + \inf_{y \in \sigma} \beta g(y)/2 < \beta g(v_m)/4 + \inf_{y \in \sigma} \beta g(y)/2\\
 &\leq \sup_{y \in \sigma} \beta g(y)/4 + \inf_{y \in \sigma} \beta g(y)/2 < \inf_{y \in \sigma} \beta g(y) \leq \beta g(\hat\sigma).
\end{align*}
Hence there is an arc $\gamma(\sigma,v_m,j) : \I \to X$ from some point of $g(\hat\sigma)$ to $z(v_m,j)$ of $\diam_d{\gamma(\sigma,v_m,j)(\I)} < \alpha g(\hat\sigma)/2$ by Lemma~\ref{arc}.
Define $h(\hat\sigma) = g(\hat\sigma) \cup \{z(v_m,j) \mid m = 1, 2 \text{ and } j = 0, \cdots, k\}$.
Note that $\card{h(\hat\sigma)} \geq k + 1$.
Moreover, $d_H(g(\hat\sigma),h(\hat\sigma)) \leq \beta g(\hat\sigma) \leq \alpha g(\hat\sigma)/2$.
Let $\phi(\sigma) : \I \to \fin(X)$ be a map defined by
 $$\phi(\sigma)(t) = g(\hat\sigma) \cup \{\gamma(\sigma,v_m,j)(t) \mid m = 1, 2 \text{ and } j = 0, \cdots, k\},$$
 which is a path from $g(\hat\sigma)$ to $h(\hat\sigma)$.
For each $m = 1, 2$, define a map $h : \langle v_m,\hat\sigma \rangle \to \fin(X)$ of the segment between $v_m$ and $\hat\sigma$ in $\sigma$ as follows:
 $$h((1 - t)v_m + t\hat\sigma) = \left\{
 \begin{array}{ll}
  g((1 - 2t)v_m + 2t\hat\sigma) \cup \{z(v_m,j) \mid j = 0, \cdots, k\} &\text{if } 0 \leq t \leq 1/2,\\
  \phi(\sigma)(2t - 1) \cup \{z(v_m,j) \mid j = 0, \cdots, k\} &\text{if } 1/2 \leq t \leq 1.
 \end{array}
 \right.$$
Then for every $y \in \sigma$, when $y = (1 - t)v_m + t\hat\sigma$, $0 \leq t \leq 1/2$,
\begin{align*}
 d_H(g(\hat\sigma),h(y)) &\leq \max\{d_H(g(\hat\sigma),g((1 - 2t)v_m + 2t\hat\sigma)),\\
 &\ \ \ \ \ \ \ \ \ \ \ \ \ \ \ \ \max\{d(z(v_m,j),g(\hat\sigma)) \mid j = 0, \cdots, k\}\}\\
 &\leq \max\{\diam_{d_H}{g(\sigma)},\beta g(\hat\sigma)\}\\
 &< \max\{\inf_{y' \in \sigma} \beta g(y')/2,\beta g(\hat\sigma)\} \leq \beta g(\hat\sigma) \leq \alpha g(\hat\sigma)/2,
\end{align*}
 and when $y = (1 - t)v_m + t\hat\sigma$, $1/2 \leq t \leq 1$,
\begin{align*}
 d_H(g(\hat\sigma),h(y)) &\leq \max\{d_H(g(\hat\sigma),\phi(\sigma)(2t - 1)),\max\{d(z(v_m,j),g(\hat\sigma)) \mid j = 0, \cdots, k\}\}\\
 &\leq \max\{\max\{\diam_{d_H}{\gamma(\sigma,v_n,j)(\I)} \mid n = 1, 2 \text{ and } j = 0, \cdots, k\},\beta g(\hat\sigma)\}\\
 &< \max\{\alpha g(\hat\sigma)/2,\beta g(\hat\sigma)\} = \alpha g(\hat\sigma)/2.
\end{align*}
Hence, due to condition (3), we have
\begin{align*}
 d_H(g(y),h(y)) &\leq d_H(g(y),g(\hat\sigma)) + d_H(g(\hat\sigma),h(y)) \leq \diam_{d_H}{g(\sigma)} + \alpha g(\hat\sigma)/2\\
 &< \inf_{y' \in \sigma} \beta g(y')/2 + \alpha g(\hat\sigma)/2 \leq \beta g(\hat\sigma)/2 + \alpha g(\hat\sigma)/2 \leq 3\alpha g(\hat\sigma)/4\\
 &\leq 3\sup_{y' \in \sigma} \alpha g(y')/4 < \inf_{y' \in \sigma} \alpha g(y') \leq \alpha g(y).
\end{align*}
Note that for each $y \in \sigma$, $h(y)$ contains $\{z(v_1,j) \mid j = 0, \cdots, k\}$ or $\{z(v_2,j) \mid j = 0, \cdots, k\}$,
 so $\card{h(y)} \geq k + 1$.

By induction, we shall construct a map $h : |K| \to \fin(X)$ such that for each $y \in \sigma \in K \setminus K^{(0)}$, $h(y) = \bigcup_{a \in A} h(a)$ for some $A \in \fin(|\sigma^{(1)}|)$.
Assume that $h$ extends over $|K^{(n)}|$ for some $n \in \N$ such that for every $y \in \sigma \in K^{(n)} \setminus K^{(0)}$, $h(y) = \bigcup_{a \in A} h(a)$ for some $A \in \fin(|\sigma^{(1)}|)$.
Take any $\sigma \in K^{(n + 1)} \setminus K^{(n)}$.
By Lemma~3.3 of \cite{CN}, there exists a map $r : \sigma \to \fin(\partial{\sigma})$ such that $r(y) = \{y\}$ for all $y \in \partial{\sigma}$.
The map $h|_{\partial{\sigma}}$ induces $\tilde{h} : \fin(\partial{\sigma}) \to \fin(X)$ defined by $\tilde{h}(A) = \bigcup_{a \in A} h(a)$.
Then we can obtain the composition $h_\sigma = \tilde{h}r : \sigma \to \fin(X)$.
It follows from the definition that $h_\sigma|_{\partial{\sigma}} = h|_{\partial{\sigma}}$.
Observe that for each $y \in \sigma$,
 $$h_\sigma(y) = \tilde{h}r(y) = \bigcup_{y' \in r(y)} h(y') = \bigcup_{y' \in r(y)} \bigcup_{a \in A(y')} h(a) = \bigcup_{a \in \bigcup_{y' \in r(y)} A(y')} h(a),$$
 where $h(y') = \bigcup_{a \in A(y')} h(a)$ for some $A(y') \in \fin(|\sigma^{(1)}|)$ by the inductive assumption.
Thus we can extend $h$ over $|K^{(n + 1)}|$ by $h|_\sigma = h_\sigma$ for all $\sigma \in K^{(n + 1)} \setminus K^{(n)}$.

After completing this induction, we can obtain a map $h : |K| \to \fin(X)$.
For each $\sigma \in K \setminus K^{(0)}$, each $y \in \sigma$ and each $a \in |\sigma^{(1)}|$, we get
\begin{align*}
 d_H(g(y),h(a)) &\leq d_H(g(y),g(a)) + d_H(g(a),h(a)) < \diam_{d_H}{g(\sigma)} + \alpha g(a)\\
 &< \inf_{y' \in \sigma} \beta g(y')/2 + \sup_{y' \in \sigma} \alpha g(y') \leq \inf_{y' \in \sigma} \alpha g(y')/4 + 4\inf_{y' \in \sigma} \alpha g(y')/3\\
 &= 19\inf_{y' \in \sigma} \alpha g(y')/12 < 2\alpha g(y).
\end{align*}
Therefore we have
\begin{align*}
 d_H(g(y),h(y)) &= d_H\Bigg(g(y),\bigcup_{a \in \bigcup_{y' \in r(y)} A(y')} h(a)\Bigg) \leq \max_{a \in \bigcup_{y' \in r(y)} A(y')} d_H(g(y),h(a))\\
 &< 2\alpha g(y),
\end{align*}
 which implies that $h$ is $\mathcal{V}$-close to $g$.
Remark that $\{z(v,j) \mid j = 0, \cdots, k\} \subset h(y)$ for some $v \in \sigma^{(0)}$,
 and hence $\card{h(y)} \geq k + 1$.
It follows that $h(|K|) \cap \fin^k(X) = \emptyset$.
Then we may replace $h(y)$ with $g(y) \cup h(y)$ for every $y \in |K|$,
 so we have $g(y) \subset h(y)$.
The rest of this proof is to show that $\cl{h(|K|)} \cap \fin^k(X) = \emptyset$.

Suppose that there exists a sequence $\{y_n\}_{n \in \N}$ of $|K|$ such that $\{h(y_n)\}_{n \in \N}$ is converges to some $A \in \fin^k(X)$.
Take the carrier $\sigma_n \in K$ of $y_n$ and choose $v_n \in \sigma_n^{(0)}$ so that $\{z(v_n,j) \mid j = 0, \cdots, k\} \subset h(y_n)$.
Since $g(y_n) \subset h(y_n)$,
 replacing $\{g(y_n)\}_{n \in \N}$ with a subsequence, we can obtain $B \subset A$ to which $\{g(y_n)\}_{n \in \N}$ converges by Lemma~\ref{subseq.}.
Then $\{\beta g(y_n)\}_{n \in \N}$ converges to $\beta(B) > 0$.
On the other hand, for every $\epsilon > 0$, there exists $n_0 \in \N$ such that if $n \geq n_0$,
 then $d_H(h(y_n),A) < \epsilon$.
Then we can choose $0 \leq i(n) < j(n) \leq k$ for each $n \geq n_0$ so that $z(v_n,i(n)), z(v_n,j(n)) \in B_d(a,\epsilon)$ for some $a \in A$ because
 $$\card{A} \leq k < k + 1 = \card\{z(v_n,j) \mid j = 0, \cdots, k\}.$$
Note that
\begin{align*}
 \beta g(y_n)/(8(k+1)) &\leq \sup_{y \in \sigma_n} \beta g(y)/(8(k+1)) < \inf_{y \in \sigma_n} \beta g(y)/(4(k+1))\\
 &\leq \beta g(v_n)/(4(k+1)) \leq d(z(v_n,i(n)), z(v_n,j(n))) < 2\epsilon,
\end{align*}
 which means that $\{\beta g(y_n)\}_{n \in \N}$ converges to $0$.
This is a contradiction.
Consequently, $\cl{h(|K|)} \cap \fin^k(X) = \emptyset$.
\end{proof}

\section{The topological type of $\fin(X)$}

In this section, we will discuss the topological type of $\fin(X)$.
Throughout this section, we assume that $\kappa$ is an infinite cardinal.
By $\ell_2^f(\kappa)$, we denote the linear subspace spanned by the canonical orthonormal basis in the Hilbert space of weight $\kappa$.
J.~Mogilski \cite{Mog} (cf.~\cite{CDM}) gave a characterization to $\ell_2^f(\aleph_0)$,
 that is extended to the uncountable case of $\kappa$ by the author \cite{Kos1} (cf.~\cite{SaY}).
Using his characterization, D.~Curtis and N.T.~Nhu \cite{CN} (cf.~\cite{Cu4}) showed the following theorem:

\begin{thm}
The hyperspace $\fin(X)$ is homeomorphic to $\ell_2^f(\aleph_0)$ if and only if $X$ is non-degenerate, connected, locally path-connected, strongly countable-dimensional\footnote{A space is said to be strongly countable-dimensional if it is written as a countable union of finite-dimensional closed subsets.} and $\sigma$-compact.
\end{thm}

We do not know how condition on $X$ is necessary and sufficient for $\fin(X)$ to be homeomorphic to $\ell_2^f(\kappa)$ for an uncountable cardinal $\kappa$.
K.~Mine, K.~Sakai and M.~Yaguchi \cite{MSY} proved the following:

\begin{thm}
If $X$ is a connected topological manifold modeled by $\ell_2^f(\kappa)$,
 then $\fin(X)$ is homeomorphic to $\ell_2^f(\kappa)$.
\end{thm}

\begin{prm}
Give a necessary and sufficient condition on a space $X$ for $\fin(X)$ to be homeomorphic to $\ell_2^f(\kappa)$ for an uncountable cardinal $\kappa$.
\end{prm}


\begin{thebibliography}{9}
 \bibitem{Cu1} D.W.~Curtis, \textit{Hyperspaces homeomorphic to Hilbert space}, Proc. Amer. Math. Soc. \textbf{75}, (1979), 126--130.
 \bibitem{Cu4} D.W.~Curtis, \textit{Hyperspaces of finite subsets as boundary sets}, Topology Appl. \textbf{22}, (1986), 97--107.
 \bibitem{CDM} D.W.~Curtis, T.~Dobrowolski and J.~Mogilski, \textit{Some applications of the topological characterizations of the sigma-compact spaces $\ell_2^f$ and $\Sigma$}, Trans. Amer. Math. Soc. \textbf{284} (1984), 837--846.
 \bibitem{CN} D.W.~Curtis and N.T.~Nhu, \textit{Hyperspaces of finite subsets which are homeomorphic to $\aleph_0$-dimensional linear metric spaces}, Topology Appl. \textbf{19}, (1985), 251--260.
 \bibitem{Kos1} K.~Koshino, \textit{Characterizing non-separable sigma-locally compact infinite-dimensional manifolds and its applications}, J. Math. Soc. Japan \textbf{66} (2014), 1155--1189.
 \bibitem{Mil3} J.~van Mill, The infinite-dimensional topology of function spaces, North-Holland Math. Library, \textbf{64}, North-Holland Publishing Co., Amsterdam, 2001.
 \bibitem{MSY} K.~Mine, K.~Sakai and M.~Yaguchi, \textit{Hyperspaces of finite sets in universal spaces for absolute Borel classes}, Bull. Pol. Acad. Sci. Math. \textbf{53}, (2005), 409--419.
 \bibitem{Mog} J.~Mogilski, \textit{Characterizing the topology of infinite-dimensional $\sigma$-compact manifolds}, Proc. Amer. Math. Soc. \textbf{92} (1984), 111--118.
 \bibitem{Sa6} K.~Sakai, Geometric Aspects of General Topology, Springer, SMM, Springer, Tokyo, 2013.
 \bibitem{SaY} K.~Sakai and M.~Yaguchi, \textit{Characterizing manifolds modeled on certain dense subspaces of non-separable Hilbert spaces}, Tsukuba J. Math. \textbf{27} (2003), 143--159.
 \bibitem{Yag} M.~Yaguchi, \textit{Hyperspaces of finite subsets of non-separable Hilbert spaces}, Tsukuba J. Math. \textbf{30}, (2006), 181--193.
\end{thebibliography}
\end{document}